\documentclass[12pt]{amsart}
\usepackage{amsfonts,amsmath,amssymb,amsbsy,amstext,amsthm}
\usepackage{color,enumerate,float,fullpage,verbatim}
\usepackage{euler, amsfonts, amssymb, latexsym, epsfig,epic}
\usepackage[shortlabels]{enumitem}

\usepackage{url}
\usepackage[colorlinks=true,hyperindex, linkcolor=Mahogany, pagebackref=false, citecolor=NavyBlue]{hyperref}
\usepackage[dvipsnames]{xcolor}
\usepackage{verbatim}
\usepackage{color}
\usepackage[all]{xy}  
\usepackage{mathrsfs}

\theoremstyle{definition}
\newtheorem{theorem}{Theorem}[section]

\numberwithin{equation}{section}
\newtheorem*{theorema}{Main theorem}

\newtheorem*{theorem*}{Theorem}

\newtheorem{lemma}[theorem]{Lemma}
\newtheorem{proposition}[theorem]{Proposition}

\newtheorem*{claim*}{Claim}

\theoremstyle{definition}
\newtheorem{definition}[theorem]{Definition}

\newtheorem{remark}[theorem]{Remark}

\newcommand{\e}{\operatorname{e}}

\newtheoremstyle{TheoremNum}
        {8pt}{8pt}              
        {\upshape}                      
        {}                              
        {\bfseries}                     
        {.}                             
        {.5em}                             
        {\thmname{#1}\thmnote{ \bfseries #3}}
  \theoremstyle{TheoremNum}

\newcommand{\m}{\mathfrak{m}}

\newcommand{\ZZ}{\mathbb{Z}}

\newcommand{\IN}{\operatorname{in}}

\newcommand{\Spec}{\operatorname{Spec}}

\newcommand{\Ext}{\operatorname{Ext}}

\newcommand{\Supp}{\operatorname{Supp}}

\newcommand{\Ass}{\operatorname{Ass}}
	
\newcommand{\IM}{\operatorname{Im}}	
\newcommand{\depth}{\operatorname{depth}}

\newcommand{\HF}{\operatorname{HF}}
\newcommand{\HS}{\operatorname{HS}}

\newcommand{\p}{\mathfrak{p}}

\newcommand{\MIN}{\operatorname{Min}}

\renewcommand{\leq}{\leqslant}
\renewcommand{\geq}{\geqslant}
\newcommand{\gin}{\operatorname{gin}}

\newcommand{\sat}{\operatorname{{sat}}}

\newcommand{\adeg}{\operatorname{adeg}}

\title{A criterion for sequential Cohen-Macaulayness}
\author{Giulio Caviglia}
\address{Department of Mathematics, Purdue University, 150 N. University Street, West Lafayette, IN 47907-2067, USA}
\email{gcavigli@purdue.edu}
\author{Alessandro De Stefani}
\address{Dipartimento di Matematica, Universit{\`a} di Genova, Via Dodecaneso 35, 16146 Genova, Italy}
\email{destefani@dima.unige.it}
\thanks{The first author was partially supported by a grant from the Simons Foundation (41000748, G.C.)}
\thanks{The second author was partially supported by the PRIN 2020 project 2020355B8Y ``Squarefree Gr{\"o}bner degenerations, special varieties and related topics''.}

\subjclass[2010]{}

\begin{document}

\begin{abstract}
The purpose of this note is to show that a finitely generated graded module $M$ over $S=k[x_1,\ldots,x_n]$, $k$ a field, is sequentially Cohen-Macaulay if and only if its arithmetic degree $\adeg(M)$ agrees with $\adeg(F/\gin_{revlex}(U))$, where $F$ is a graded free $S$-module and $M \cong F/U$. This answers positively a conjecture of Lu and Yu from 2016.
\end{abstract} 

\maketitle

\section{Introduction}
Sequentially Cohen-Macaulay modules were first introduced in the graded setting by Stanley \cite{Stanley}, in connection with the theory of Stanley-Reisner rings. Later, Schenzel \cite{Schenzel} introduced the notion of Cohen-Macaulay filtered module in a more general setting, and showed that in the graded case it coincides with the one introduced by Stanley.

Let $k$ be a field, and $S=k[x_1,\ldots,x_n]$ with the standard grading. Let $M$ be a finitely generated $\ZZ$-graded $S$-module of Krull dimension $d$. A module $M$ is sequentially Cohen-Macaulay if one of the following three equivalent conditions is satisfied:
\begin{enumerate}
\item (Stanley) There exists a filtration $0 = M_0 \subsetneq M_1 \subsetneq M_2 \subsetneq \ldots \subsetneq M_t = M$ with $M_i/M_{i-1}$ Cohen-Macaulay graded $S$-modules of dimension $d_i$ satisfying $d_i>d_{i-1}$ for all $i=1,\ldots,t$.
\item (Schenzel) If $\delta_i(M)$ denotes the largest submodule of $M$ of dimension less than or equal to $i$, then the filtration 
\[
0 \subseteq \delta_0(M) \subseteq \delta_1(M) \subseteq \ldots \subseteq \delta_{d}(M) = M
\]
is such that $\delta_i(M)/\delta_{i-1}(M)$ is either zero or Cohen-Macaulay for all $i=1,\ldots,d$.
\item (Peskine) For every $i=0,\ldots,n$ the module $\Ext^{n-i}_S(M,S)$ is either zero or Cohen-Macaulay of dimension $i$.
\end{enumerate}
Over the years sequentially Cohen-Macaulay modules have been studied extensively, even from very different perspectives (for instance, see \cite{Duval,HerzogHibi,HerzogReinerWelker,ABG,Goodarzi,Josep,CuongCuong,CT}). We invite the interested reader to consult \cite{sCMSurvey} for further details.

For the purposes of this note, it is important to mention a result of Herzog and Sbarra, which characterizes sequentially Cohen-Macaulay modules. We state it here in a dual version:
\begin{theorem}\cite[Main theorem]{HerzogSbarra} \label{HS}
Let $F$ be a finitely generated free $S$-module, and $U \subseteq F$ be a homogeneous submodule. The module $F/U$ is sequentially Cohen-Macaulay if and only if $\HF(\Ext^{n-i}_S(F/U,S)) = \HF(\Ext^{n-i}_S(F/\gin_{revlex}(U),S))$ for all $i=0,\ldots,n$, where $\HF(-)$ denotes the Hilbert function of a graded $S$-module.
\end{theorem}
This theorem has been extended in two different directions: weaker versions of sequential Cohen-Macaulayness can be characterized either by showing that equality of Hilbert functions holds only for certain cohomological indices \cite{SbarraStrazzanti}, or by replacing revlex with certain partial revlex orders \cite{CDS}.

To state our main theorem, we recall the notion of arithmetic degree, introduced by Bayer and Mumford in \cite{BM} (see also \cite{STV,Vasconcelos}). First, recall that if $0\ne M$ is a finitely generated $\ZZ$-graded module of dimension $d$, and $\HF(M;j) = \dim_k(M_j)$ is the Hilbert function of $M$ in degree $j$, then for $j \gg 0$ we have
\[ 
\sum_{i \leq j} \HF(M;i) = \frac{\e(M)}{d!} j^{d} +  O(j^{d-1})
\]
for some $\e(M) \in \ZZ_{>0}$, called multiplicity (or degree) of $M$. Given an integer $r \geq d$, we let 
\[
\e_r(M) = \begin{cases} \e(M) & \text{ if } r=d \\ 0 & \text{ otherwise. }
\end{cases}
\]


\begin{definition} Let $S=k[x_1,\ldots,x_n]$, and $M$ be a finitely generated graded $S$-module. For all $r=0,\ldots,n$ we let $\adeg_r(M) = \e_r(\Ext^{n-r}_S(M,S))$. The arithmetic degree of $M$ is defined as $\adeg(M) = \sum_{r=0}^n \adeg_r(M)$.
\end{definition}
Recall that $\dim(\Ext_S^{n-r}(M,S)) \leq r$ always holds (see for instance \cite[Section 3.1]{Schenzel_Dual}), therefore the above definition makes sense. Now let $F$ be a free $S$-module, and $U \subseteq F$ be a graded submodule. Given any weight $\omega \in \ZZ^n$, by upper semi-continuity one has $\adeg_r(F/U) \leq \adeg_r(F/\IN_\omega(U))$ for all $r=0,\ldots,n$ (see \cite{STV} for the case of cyclic modules); in particular, $\adeg(F/U) \leq \adeg(F/\gin_{revlex}(U))$ always holds. If $F/U$ is sequentially Cohen-Macaulay, then Theorem \ref{HS} immediately gives that $\adeg(F/U) = \adeg(F/\gin_{revlex}(U))$. This result was observed by Lu and Yu \cite[Proposition 3.6]{LY}, and in the same paper they conjecture that the converse holds as well. The purpose of this note is to prove their conjecture:
\begin{theorema} Let $S=k[x_1,\ldots,x_n]$, with the standard grading. Let $F$ be a finitely generated graded free $S$-module and $U \subseteq F$ be a homogeneous submodule. We have that $F/U$ is sequentially Cohen-Macaulay if and only if $\adeg(F/U) = \adeg(F/\gin_{revlex}(U))$.
\end{theorema}
Actually, we prove their conjecture in a bigger generality, by showing that it is not necessary to take general coordinates when computing initial ideals as long as $x_n,\ldots,x_1$ is a filter regular sequence for $F/U$ (see Section \ref{section main}). As a consequence of our proof, we also obtain an analogous generalization of Theorem \ref{HS}.

\section{Preliminaries and main result} \label{section main}
Let $k$ be a field, and $S=k[x_1,\ldots,x_n]$ be a polynomial ring with standard grading $\deg(x_i)=1$ for all $i$. Let $\m=(x_1,\ldots,x_n)$. Throughout, $M$ will always denote a finitely generated $\ZZ$-graded $S$-module.

Consider the weight $\omega=(0,\ldots,0,-1) \in \ZZ^n$ and, for a homogeneous submodule $U$ of a graded free $S$-module $F$ we let $\IN(U):= \IN_\omega(U) \subseteq F$ be the initial submodule. Since revlex can be obtained as $\IN_\Omega$ for the following matrix of weights
\[
\Omega = \begin{bmatrix} 0 & 0 & \ldots & 0& -1 \\
0 & & 0 \ldots & -1 & -1 \\
\vdots & \vdots & \vdots & \vdots & \vdots \\
0 & -1 & \ldots & -1 & -1 
\end{bmatrix},
\]
of which $\omega$ is the first row, we call $\IN(U)$ a ``partial revlex submodule''. See \cite{CDS} for more details on this construction, where $\IN(-)$ is denoted as $\IN_{rev_{1}}(-)$.

\begin{definition} Let $M$ be a finitely generated $\ZZ$-graded $S$-module. A homogeneous element $f \in S$ is called filter regular if $0:_M f$ has finite length. A sequence of elements $f_1,\ldots,f_t$ is called a filter regular sequence if $f_{i+1}$ is filter regular for $M/(f_1,\ldots,f_i)M$ for all $i<t$.
\end{definition}

\begin{remark} \label{remark partial} Let $U$ be a homogeneous submodule of a graded free $S$-module $F$. 
Since $\IN(U)$ is a partial deformation towards $\IN_{revlex}(U)$, by upper semi-continuity one has that $\adeg_r(F/U) \leq \adeg_r(F/\IN(U)) \leq \adeg_r(F/\IN_{revlex}(U))$ for all $r=0,\ldots,n$.
\end{remark}

We will need the following lemma in the proof of the main theorem. While such a result might be well-known in the literature, we include a proof for sake of completeness.

\begin{lemma} \label{multiplicity}
Let $M$ be a finitely generated $\ZZ$-graded $S$-module of dimension $d>0$, and let $\ell \in S_1$ be such that $M/\ell M$ is Cohen-Macaulay of dimension $d-1$, and $\e_{d-1}(M/\ell M) = \e_d(M)$. Then $\ell$ is an $M$-regular element, and $M$ is Cohen-Macaulay.
\end{lemma}
\begin{proof}
Without loss of generality we may assume that $k$ is infinite. Then, we can find linear forms $\ell_1,\ldots,\ell_{d-1}$ which form a filter regular sequence for $M$ and which form a full regular sequence for $M/\ell M$. If we consider the one-dimensional module $N=M/(\ell_1,\ldots,\ell_{d-1})M$, the first condition guarantees that $\e_d(M) = \e_1(N)$. Since $M/(\ell,\ell_1,\ldots,\ell_{d-1})M = N/\ell N$ is zero-dimensional, we have that $\ell$ is a parameter for $N$. Using that $M/\ell M$ is Cohen-Macaulay, we conclude that $\e_1(N) = \e_d(M) = \e_{d-1}(M/\ell M) = \lambda(N/\ell N)$, where $\lambda(-)$ denotes the length of a module. Since $N$ is a finitely generated $k[\ell]$-module, by the structure theorem for modules over a PID we can write $N \cong k[\ell]^{\oplus \e_1(N)} \oplus T$. It follows that $\lambda(N/\ell N) = \e_1(N) + \lambda(T/\ell T)$, and therefore $T/\ell T = 0$. We conclude that $T  = 0$ by Nakayama's Lemma.

So far we have shown that $\ell$ is $N$-regular and, in particular, $N$ is a one-dimensional Cohen-Macaulay module. If $d>1$, to conclude we repeatedly use the following version of Sally's machine:
\begin{claim*} Let $B$ be a finitely generated $\ZZ$-graded $S$-module of dimension $d>1$, and let $x \in S_1$ be a filter regular element for $B$. If $\depth(B/x B)>0$, then $x$ is $B$-regular. In particular, $\depth(B) = \depth(B/xB)+1$.
\end{claim*}
{\it Proof of the Claim.} First, the short exact sequence $0 \to B/(0:_B x)[-1] \to B \to B/xB \to 0$ yields an exact sequence $0 \to H^0_\m(B/(0:_B x))[-1] \stackrel{\varphi}{\to} H^0_\m(B) \to H^0_\m(B/xB)$. Since $x$ is filter regular for $B$ we have
\[
H^0_\m\left(\frac{B}{0:_B x}\right) = \frac{(0:_B x):x^\infty}{0:_B x}  = \frac{0:_Bx^\infty}{0:_Bx} = \frac{H^0_\m(B)}{(0:_Bx)},
\]
and it follows that $\IM(\varphi) = x H^0_\m(B)$. By assumption we have that $H^0_\m(B/xB)=0$, and therefore we obtain that $xH^0_\m(B) = H^0_\m(B)$. We then conclude by Nakayama's lemma that $H^0_\m(B)=0$, and since $x$ is filter regular, it is $B$-regular. In particular, it follows that $\depth(B) = \depth(B/xB)+1$. 
\end{proof}

For a module $M$ we let $\Ass_r(M) = \Ass(M) \cap \{\p \in \Spec(S) \mid \dim(S/\p) = r\}$. 
\begin{remark} \label{remark Ass} We recall that, if $M$ is finitely generated and $\ZZ$-graded, then $\Ass_r(M) = \Ass_r(Ext^{n-r}_S(M,S))$ for all $r = 0,\ldots,n$. In fact, this follows by noticing that $\p \in \Ass_r(M)$ if and only if $H^0_{\p S_\p}(M_\p) \ne 0$, if and only if $\left(\Ext^{n-r}_S(M,S)\right)_\p \ne 0$, if and only if $\p \in \MIN(\Supp(\Ext^{n-r}_S(M,S)))$, if and only if $\p \in \Ass_r(\Ext^{n-r}_S(M,S))$ (see also \cite[Proposition A.3]{Faridi} and \cite[Section 3.1]{Schenzel_Dual}). 
\end{remark}

\begin{remark} \label{in sequentially} Let $F$ be a finitely generated free $S$-module, and $U \subseteq F$ be a homogeneous submodule. It is clear from \cite{HerzogSbarra} that $F/\gin_{revlex}(U)$ is sequentially Cohen-Macaulay. However, the same argument works for $F/\IN_{revlex}(U)$ assuming that $x_n,x_{n-1},\ldots,x_1$ forms a filter regular sequence for $F/U$ (see also \cite[Key Example 2.15]{CDS}). In particular, under these assumptions, for all $r=0,\ldots,n$ the module $\Ext^{n-r}_S(F/\IN_{revlex}(U),S)$ is either zero or Cohen-Macaulay of dimension $r$.
\end{remark}

We are now ready to prove the conjecture of Lu and Yu \cite{LY}. 

\begin{theorem} \label{main} Let $S=k[x_1,\ldots,x_n]$, with the standard grading. Let $F$ be a finitely generated graded free $S$-module and $U \subseteq F$ be a homogeneous submodule. Assume that $x_n,x_{n-1},\ldots,x_1$ is a filter regular sequence for $F/U$. The following are equivalent:
\begin{enumerate}
\item $F/U$ is sequentially Cohen-Macaulay,
\item $\HF(\Ext^{n-r}_S(F/U,S)) = \HF(\Ext^{n-r}_S(F/\IN_{revlex}(U),S))$ for all $r=0,\ldots,n$,
\item $\adeg(F/U) = \adeg(F/\IN_{revlex}(U))$.
\end{enumerate}
\end{theorem}
\begin{proof}
We prove the equivalence of the three conditions by induction on $d=\dim(M)$. Let $V=\IN_{revlex}(U)$. First of all, observe that $U^{\sat} = U:x_n^\infty$, and that $\IN_{revlex}(U^{\sat}) = V:x_n^\infty = V^{\sat}$ by well-known properties of revlex-type orders (see \cite[15.7]{Eisenbud}). It follows that $\HF(U^{\sat}/U) = \HF(V^{\sat}/V)$. Since $U^{\sat}/U$ is the graded Matlis dual of $\Ext^n_S(F/U,S)$ (and similarly for $V$), we conclude that $\HF(\Ext^n_S(F/U,S)) = \HF(\Ext^n_S(F/V,S))$. In particular, $\adeg_0(F/U) = \e_0(\Ext^n_S(F/U,S)) = \e_0(\Ext^n_S(F/V,S)) =  \adeg_0(F/V)$.

Since modules of dimension zero are automatically sequentially Cohen-Macaulay, for $d=0$ the three statements are trivially equivalent because they are all true.

Now assume that $d>0$. Recall that $F/U$ is sequentially Cohen-Macaulay if and only if so is $F/U^{\sat}$ \cite[Corollary 2.8]{sCMSurvey}, and that $\Ext^{n-r}_S(F/U,S) \cong \Ext^{n-r}_S(F/U^{\sat},S)$ for all $r>0$. Similar considerations hold for $V$. In view of the previous equalities, we may assume that $\depth(F/U)>0$ after possibly replacing $U$ with $U^{\sat}$. 

Assume (1). By Remark \ref{remark Ass} we have that $\Ass_r(M)= \Ass_r(\Ext^{n-r}_S(F/U,S))$, and moreover $\Ass_r(\Ext^{n-r}_S(F/U,S)) = \Ass(\Ext^{n-r}_S(F/U,S))$ because $\Ext^{n-r}_S(F/U,S)$ is either zero or Cohen-Macaulay of dimension $r$. Since $x_n$ is filter regular for $F/U$ we conclude that it is regular for $\Ext^{n-r}_S(F/U,S)$ for all $r>0$. For all $r>0$ we have short exact sequences
\[
\xymatrix{
0 \ar[r] & \Ext^{n-r}_S(F/U,S) \ar[r]^-{\cdot x_n} & \Ext^{n-r}_S(F/U,S)[+1] \ar[r] & \Ext^{n-r+1}_S(F/(U+x_nF),S) \ar[r] & 0,
}
\]
which give that $F/(U+x_nF)$ is sequentially Cohen-Macaulay of dimension $d-1$. Moreover, for all $r>0$ and $j \in \ZZ$ we obtain
\begin{equation}
\label{HF}
\HF(\Ext^{n-r+1}_S(F/(U+x_nF),S);j)  = \HF(\Ext^{n-r}_S(F/U,S);j+1) - \HF(\Ext^{n-r}_S(F/U,S);j).
\end{equation}
Now let $\overline{S} = k[x_1,\ldots,x_{n-1}]$. We can identify $F/(U+x_nF)$ with a sequentially Cohen-Macaulay $\overline{S}$-module $\overline{F}/\overline{U}$, where $\overline{F}$ is  a graded free $\overline{S}$-module and $\overline{U} \subseteq \overline{F}$ is a homogeneous submodule. Since $\dim(\overline{F}/\overline{U}) = d-1$, by induction we have that $\HF\left(\Ext^{n-r}_{\overline{S}}(\overline{F}/\overline{U},\overline{S})\right) = \HF\left(\Ext^{n-r}_{\overline{S}}(\overline{F}/\IN_{revlex}(\overline{U}),\overline{S})\right)$ for all $r > 0$. By \cite[Lemma 3.1.16]{BrunsHerzog}, for all $r>0$ we have that 
\[
\Ext^{n-r}_{\overline{S}}(\overline{F}/\overline{U},\overline{S})\cong \Ext^{n-r+1}_S(F/(U+x_nF),S)
\]
and 
\[
\Ext^{n-r}_{\overline{S}}(\overline{F}/\IN_{revlex}(\overline{U}),\overline{S})\cong \Ext^{n-r+1}_S(F/(V+x_nF),S).
\]
Using that $\IN_{revlex}(U+x_nF) = V+x_nF$ (see \cite[15.7]{Eisenbud}), we obtain
\[
\HF(\Ext^{n-r+1}_S(F/(U+x_nF),S)) = \HF(\Ext^{n-r+1}_S(F/\IN_{revlex}(U+x_nF),S)) \text{ for all } r>0.
\]

By the proof of \cite[Proposition 3.5]{CDS} and by \cite[Lemma 2.9]{CDS}, for all $r>0$ we have short exact sequences
\begin{equation}
\label{ses revlex}
\xymatrix{
0 \ar[r] & \Ext^{n-r}_S(F/V,S) \ar[r]^-{\cdot x_n} & \Ext^{n-r}_S(F/V,S)[+1] \ar[r] & \Ext^{n-r+1}_S(F/V+x_nF,S) \ar[r] & 0,
}
\end{equation}
which give
\begin{equation}
\label{HF1}
\HF(\Ext^{n-r+1}_S(F/V+x_nF,S);j)  = \HF(\Ext^{n-r}_S(F/VS);j+1) - \HF(\Ext^{n-r}_S(F/V,S);j).
\end{equation}
Combining (\ref{HF}) and (\ref{HF1}), together with the fact that the $\Ext$ modules vanish for sufficiently negative degrees, we finally obtain
\[
\HF(\Ext^{n-r}_S(F/U,S)) = \HF(\Ext^{n-r}_S(F/V,S)) \text{ for all } r>0,
\]
and this concludes the proof that (1) $\Rightarrow$ (2). The fact that (2) $\Rightarrow$ (3) is trivial. 

We finally show that (3) $\Rightarrow$ (1). Let $\omega=(0,\ldots,0,-1) \in \ZZ^n$, and let $W=\IN_\omega(U)$. By Remark \ref{remark partial} and our assumptions we must have $\adeg_r(F/U)=\adeg_r(F/W) = \adeg_r(F/V)$ for all $r=0,\ldots,n$. Again using the proof of \cite[Proposition 3.5]{CDS} and \cite[Lemma 2.9]{CDS}, for all $r>0$ we obtain short exact sequences
\begin{equation}
\label{ses}
\xymatrix{
0 \ar[r] & \Ext^{n-r}_S(F/W,S) \ar[r]^-{\cdot x_n} & \Ext^{n-r}_S(F/W,S)[+1] \ar[r] & \Ext^{n-r+1}_S(F/(U+x_nF),S) \ar[r] & 0,
}
\end{equation}
where we use that $W+x_nF = U+x_nF$ holds because $\omega$ is a partial revlex order. The short exact sequences (\ref{ses}) and (\ref{ses revlex}) give that $\adeg_{r-1}(F/(U+x_nF)) = \adeg_r(F/W) = \adeg_r(F/V) = \adeg_{r-1}(F/(V+x_nF))$. 
We therefore obtain that $\adeg(F/(U+x_nF)) = \adeg(F/(V+x_nF))$. As above, we can identify $F/(U+x_nF)$ with an $\overline{S}$-module $\overline{F}/\overline{U}$. 
In this way, $F/(V+x_nF)$ can be identified with $\overline{F}/\IN_{revlex}(\overline{U})$, and therefore we have that $\adeg(\overline{F}/\overline{U}) = \adeg(\overline{F}/\IN_{revlex}(\overline{U}))$. Since $\dim(\overline{F}/\overline{U})=d-1$ and $x_{n-1},\ldots,x_1$ forms a filter regular sequence for $\overline{F}/\overline{U}$, by induction we have that $\overline{F}/\overline{U}$ is sequentially Cohen-Macaulay, and so is $F/(U+x_nF)$.

Now we note that if $\Ext^{n-r}_S(F/U,S) \ne 0$ then $\Ext^{n-r}_S(F/V,S) \ne 0$ by upper semi-continuity. Since we have $\e_r(\Ext^{n-r}_S(F/U,S)) = \adeg_r(F/U) = \adeg_r(F/V) = \e_r(\Ext^{n-r}_S(F/V,S))$, and the latter is positive by Remark \ref{in sequentially}, it follows that $\dim(\Ext^{n-r}_S(F/U,S))=r$. To complete the proof, we show by induction on $r>0$ that if $\Ext^{n-r}_S(F/U,S) \ne 0$ then it is a Cohen-Macaulay module, and that $x_n$ is a regular element for it.

First note that for $r>0$, since $\Ass_r(F/U) = \Ass_r(\Ext^{n-r}_S(F/U,S))$ by Remark \ref{remark Ass}, and because $x_n$ is filter regular for $F/U$, we have that $x_n$ avoids all minimal primes of $\Supp(\Ext^{n-r}_S(F/U,S))$ of dimension $r$. In particular, $\dim(\Ext^{n-r}_S(F/U,S) \otimes_S S/(x_n)) = r-1$.

For $r=1$ we have an exact sequence
\[
\xymatrix{
\Ext^{n-1}_S(F/U,S) \ar[r]^-{\cdot x_n} & \Ext^{n-1}_S(F/U,S)[+1] \ar[r] & \Ext^n(F/(U+x_nF),S) \ar[r] & 0,
}
\]
which gives that $\Ext^n_S(F/(U+x_nF),S) \cong \Ext^{n-1}_S(F/U,S) \otimes_S S/(x_n)$. Using the short exact sequence (\ref{ses}), and recalling that $U+x_nF = W + x_nF$, we obtain 
\begin{align*}
\e_1(\Ext^{n-1}_S(F/U,S)) & = \e_1(\Ext^{n-1}_S(F/W,S)) \\
& =  \e_0(\Ext^n_S(F/(U+x_nF),S)) = \e_0(\Ext^{n-1}_S(F/U,S) \otimes_S S/(x_n)).
\end{align*}
By Lemma \ref{multiplicity} we conclude that $\Ext^{n-1}_S(F/U,S)$ is Cohen-Macaulay, and $x_n$ is regular for it.

Now assume that $\Ext^{n-i}_S(F/U,S)$ is Cohen-Macaulay for all $i=1,\ldots,r-1$ and that $x_n$ is a regular element for all such modules. In particular, it is regular for $\Ext^{n-r+1}_S(F/U,S)$, and therefore we have an exact sequence
\[
\xymatrix{
\Ext^{n-r}_S(F/U,S) \ar[r]^-{\cdot x_n} & \Ext^{n-r}_S(F/U,S)[+1] \ar[r] & \Ext^{n-r+1}(F/(U+x_nF),S) \ar[r] & 0.
}
\]
We conclude that $\Ext^{n-r+1}_S(F/(U+x_nF),S) \cong \Ext^{n-r}(F/U,S) \otimes_S S/(x_n)$. As before, we have 
\begin{align*}
\e_r(\Ext^{n-r}_S(F/U,S)) &= \e_r(\Ext^{n-r}_S(F/W,S)) \\
 & = \e_{r-1}(\Ext^{n-r+1}_S(F/(U+x_nF),S)) \\
 & = \e_{r-1}(\Ext^{n-r}_S(F/U,S) \otimes_S S/(x_n)).
 \end{align*}
Finally, by Lemma \ref{multiplicity} we conclude that $\Ext^{n-r}_S(F/U,S)$ is Cohen-Macaulay and $x_n$ is regular for it, and the proof is complete.
\end{proof}

\bibliographystyle{abbrv}
\bibliography{References}
\end{document}